\documentclass[12pt]{amsart}       
\usepackage{amssymb}    
\usepackage{graphicx}
\usepackage{ytableau}
\usepackage{comment}
\usepackage{pgf,tikz}
\usetikzlibrary{calc}
\usetikzlibrary{arrows, automata}
\usepackage{subcaption} 
\usepackage{mathtools} 
\theoremstyle{definition} 
\newtheorem{lemma}{Lemma}[section] 
\newtheorem{corollary}[lemma]{Corollary} 

\newtheorem{theorem}[lemma]{Theorem} 
 
\newtheorem{proposition}[lemma]{Proposition}
\newtheorem{remark}[lemma]{Remark} 

\usepackage[colorlinks=true,linkcolor=blue,citecolor=magenta]{hyperref}

\makeatletter
\@namedef{subjclassname@2020}{\textup{2020} Mathematics Subject Classification}     
\makeatother

\title[Enumerating rectangle tilings by straight polyominoes]{Generating functions for straight polyomino tilings of narrow rectangles}  
\author{Mudit Aggarwal}
\address{Indraprastha Institute of Information Technology Delhi (IIIT-Delhi), New Delhi 110020, India.}
\email{mudit19063@iiitd.ac.in}
\author{Samrith Ram}
\address{Indraprastha Institute of Information Technology Delhi (IIIT-Delhi), New Delhi 110020, India.}
\email{samrith@iiitd.ac.in}



\begin{document}
\begin{abstract}
Let $m,k$ be fixed positive integers. Determining the generating function for the number of tilings of an $m\times n$ rectangle by $k\times 1$ rectangles is a long-standing open problem to which the answer is only known in certain special cases. We give an explicit formula for this generating function in the case where $m<2k$. This result is used to obtain the generating function for the number of tilings of an $m\times n \times k$ box with $k\times k\times 1$ bricks. 
\end{abstract}
\maketitle
\section{Introduction and main results}
 We consider the problem of enumerating the number of tilings of an $m\times n$ rectangle with $k\times 1$ tiles. An example of such a tiling for $k=3,m=5$ and $n=6$ is shown in Figure \ref{fig:example}.
\begin{figure}[!ht]
  \centering
  \includegraphics{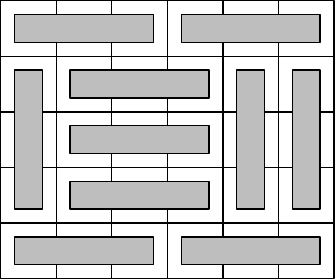}
  \caption{Tiling a $5\times 6$ rectangle with $3\times 1$ tiles.}
  \label{fig:example}
\end{figure}
It is known by a theorem of Klarner~\cite[Thm. 5]{MR248643} that such tilings exist if and only if $k$ divides either $m$ or $n$. Let $h_k(m,n)$ denote the number of such tilings. A beautiful result of Kasteleyn \cite{MR153427} and Temperley-Fisher \cite{MR136398} gives an explicit formula for $h_2(m,n)$:   
\begin{equation}
  \label{eq:dimers}
  h_2(m,n)= \prod_{j=1}^{\lceil \frac{m}{2}\rceil} \prod_{k=1}^{\lceil \frac{n}{2}\rceil}\left( 4 \cos^2\frac{j\pi}{m+1}+4\cos^2\frac{k\pi}{n+1}\right).
\end{equation}
For arbitrary integers $m,n$ no such formula is known for $h_k(m,n)$ for any integer $k\geq 3$. However, for fixed values of $m$ and $k$, the generating function $H_{k,m}(x)=\sum_{n\geq 0}h_k(m,n)x^n$ can be obtained by using the transfer-matrix method (see Stanley \cite[Sec. 4.7]{Stanley2012}) which expresses the number of tilings as the number of walks between two vertices in a suitably defined digraph. 
Klarner and Pollack \cite{MR588907} computed $H_{2,m}(x)$ for $m\leq 8$ and 
 gave an algorithm to compute polynomials $P_m$ and $Q_m$ such that $H_{2,m}(x)=P_m(x)/Q_m(x)$. Hock and McQuistan \cite{MR739603} found recurrences for $h_2(m,n)$ in terms of $n$ for each integer $m\leq 10$. Stanley~\cite{MR798013} used the explicit formula \eqref{eq:dimers} above to determine the degrees of the polynomials $P_m(x)$ and $Q_m(x)$ and prove several properties of the polynomials $P_m(x)$ and $Q_m(x)$. 
 By using the transfer-matrix approach Mathar \cite{mathar2014} derived several generating functions which enumerate tilings of the $m\times n$ rectangle with $a\times b$ tiles for fixed values of $a,b$. One of the drawbacks of the transfer-matrix approach is that  it involves the evaluation of large determinants for large values of the parameters.

For $k\geq 3$, there does not appear to have been substantial progress on the general problem of computing $H_{k,m}(x)$. In this paper we compute $H_{k,m}(x)$ for arbitrary positive integers $m$ and $k$ satisfying $m<2k$. The cases where $m\leq k$ are somewhat trivial but the case $k<m<2k$ is more interesting and we prove (see Theorem \ref{th:main1}) that the generating function is surprisingly simple:
\begin{align}
  \sum_{n\geq 0}h_k(m,n)x^n= \frac{(1-x^k)^{k-1}}{(1-x^k)^k - (m-k+1)x^k}.\label{eq:hk}
\end{align}
The generating function above allows for recursive computation of $h_k(m,n)$ and asymptotic estimates for large values of $n$ can be obtained by considering the smallest positive root of the denominator (see Flajolet and Sedgewick \cite[Ch. IV]{MR2483235}). In Theorem \ref{th:vert} we obtain a refinement of Equation \eqref{eq:hk} by proving that if $b_k(m,n,r)$ denotes the number of tilings in which precisely $r$ tiles are vertical, then  
$$
\sum_{n,r\geq 0}b_k(m,n,r)x^ny^r=  \frac{(1-x^k)^{k-1}}{(1-x^k)^k - (m-k+1)x^ky^k}.
$$
     
Under the same constraints on $m$, our results can also be used to derive the generating function (see Theorem \ref{th:3d}) for $\bar{h}_k(m,n)$, the number of tilings of an $m\times n\times k$ cuboid with $k\times k\times 1$ bricks:   
\begin{align*}
  \sum_{n\geq 0}\bar{h}_k(m,n)x^n = \frac{(1-2x^k)^{k-1}}{(1-2x^k)^{k-1}[1 - (m-k+2)x^k] - (m-k+1)x^k}.
\end{align*}
In this case we also obtain a refined multivariate generating function which accounts for tilings with a specific number of tiles in a given orientation. More precisely, we prove (see Corollary \ref{cor:vert3d}) that if $d_k(m,n,r,s)$ denotes the number of tilings of an $m\times n\times k$ cuboid with $k\times k\times 1$ bricks which contain precisely $r$ bricks parallel to the $yz$-plane and $s$ bricks parallel to the $xy$-plane, then
    \begin{align*}
      \sum_{n,r,s\geq 0}d_k(m,n,r,s)x^ny^rz^s=\hspace*{3in} \nonumber\\ \frac{(1-x^k-x^kz^k)^{k-1}}{\left(1-x^k-(m-k+1)x^kz^k\right)(1-x^k-x^kz^k)^{k-1}-(m-k+1)x^ky^k}.
    \end{align*}  
The key idea of our approach is to enumerate fault-free tilings which are explained in the next section. 
\section{Fault-free tilings}
Fix, throughout this paper, a positive integer $k\geq 2$. It will be convenient to introduce a coordinate system with the origin at the bottom left corner of the $m\times n$ rectangle such that a side of length $m$ of the rectangle is along the $y$-axis. Each of the unit squares, or cells, of the $m\times n$ rectangle is then represented by a pair $(r,s)$ corresponding to the coordinates of its top right corner. For instance the cells $(1,1)$ and $(3,4)$ are shaded in Figure \ref{fig:coordinates}. 
\begin{figure}[!ht]
  \centering
    \includegraphics{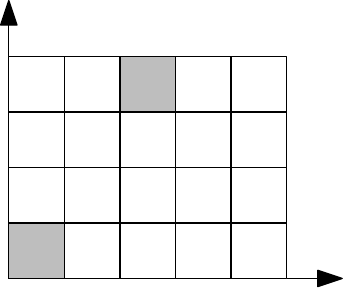}
  \caption{The cells $(1,1)$ and $(3,4)$ of a $4\times 5$ rectangle.} 
  \label{fig:coordinates}
\end{figure}

A tiling of the $m\times n$ rectangle is said to have a \emph{fault} at $x=a$ for some $1\leq a\leq n-1$ if the line $x=a$ does not intersect the interior of any tile. For instance, the tiling in Figure \ref{fig:fault} has only one fault at $x=1$ while the tiling in Figure \ref{fig:nofault} has no faults; such a tiling is called \emph{fault-free}. It is easily seen that if a tiling has $l$ faults then it can be decomposed uniquely into $l+1$ fault-free tilings. Let $a(m,n)$ denote the number of fault-free tilings of an $m\times n$ rectangle with $k\times 1$ tiles. The following well-known result \cite[Thm. 2.2.1.3]{MR3409342} relates all tilings $h(m,n)=h_k(m,n)$ to fault-free tilings.  
\begin{figure}[h]
\centering
\minipage{0.45\textwidth}
\centering
  \includegraphics{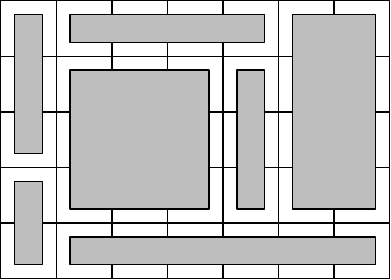} 
  \captionof{figure}{A fault at $x=1$.} 
  \label{fig:fault}
\endminipage 
\qquad
\minipage{0.45\textwidth}
\centering
  \includegraphics{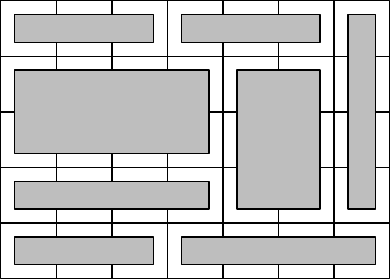}
  \captionof{figure}{A fault-free tiling.}    
  \label{fig:nofault}
  \endminipage
\end{figure}
\begin{lemma}
  \label{lem:gfs}
  If $H(x)=\sum_{n\geq 0}h(m,n)x^n$ and $A(x)=\sum_{n\geq 0}a(m,n)x^n$, then
  \begin{align*}
    H(x)=\frac{1}{1-A(x)}.
  \end{align*}
\end{lemma}
\begin{proof}
  Condition on the least positive integer $l$ such that the line $x=l$ does not intersect the interior of any tile; the number of tilings for a given $l$ is $a(m,l)h(m,n-l)$. Summing over possible values of $l$, we obtain 
  \begin{align*}
h(m,n)=\sum_{l=1}^n a(m,l)h(m,n-l) \quad (n\geq 1, \; h(m,0)=1).
  \end{align*}
  In terms of generating functions, the recurrence above reads $H(x)=A(x)H(x)+1$ which is equivalent to the statement of the lemma.
\end{proof}
For $m<k$ it is clear that $h(m,n)=1$ when $n$ is a multiple of $k$ and 0 otherwise. For $m=k$ we have the following proposition.
\begin{proposition}\label{prop:mequalsk}
For $k>1$, we have
  \begin{align*}
    \sum_{n\geq 0 }h(k,n)x^n=\frac{1}{1-x-x^k}.
  \end{align*}
\end{proposition}
\begin{proof}
  Here $a(k,1)=1=a(k,k)$ while $a(k,n)=0$ for $n\notin \{1,k\}$ (see Figure \ref{fig:only2}). Thus the generating function for fault-free tilings is $A(x)=x+x^k$ and the proposition follows from Lemma~\ref{lem:gfs}.
  \begin{figure}[h]
\centering
\includegraphics[scale=.25]{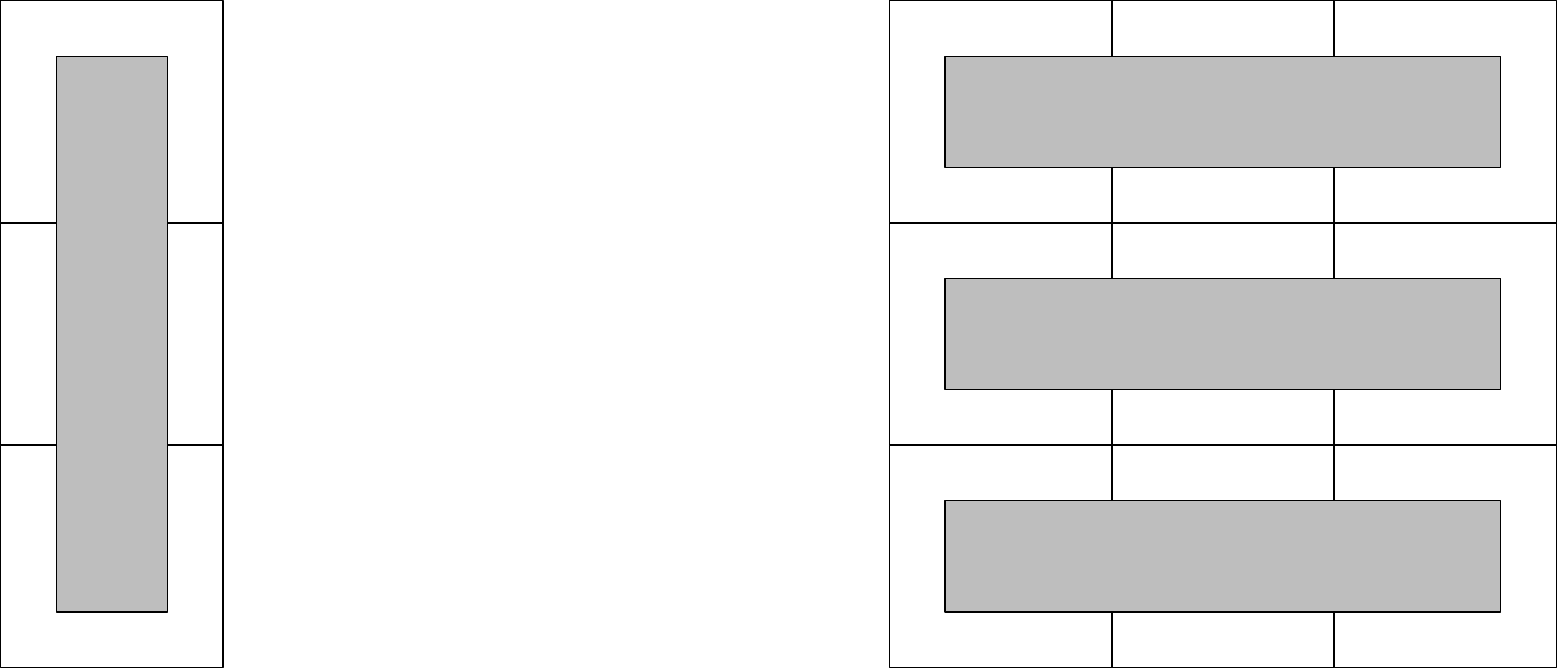}  
\caption{All possible fault-free tilings for \(m = k = 3\).}   
\label{fig:only2}     
\end{figure}
\end{proof}
\begin{remark}
  \label{rem:compositions}
  It follows from Proposition \ref{prop:mequalsk} that tilings of a $k\times n$ rectangle with $k\times 1$ tiles are in bijection with compositions of $n$ (i.e. tuples $(n_1,\ldots,n_r)$ of positive integers with $\sum n_i=n$) in which all parts are equal to 1 or $k$. We will require this fact later on. 
\end{remark}

We now consider the case $k<m<2k$. The following lemma is the key to computing fault-free tilings in this case.
\begin{lemma}
  \label{lem:main}
Suppose $k<m<2k$ and $n>k$. In any fault-free tiling of an $m\times n$ rectangle by $k\times 1$ tiles there exist $k$ contiguous rows such that all tiles in the remaining $m-k$ rows are horizontal.
\end{lemma}

\begin{proof}
  A cell $(i,j)$ is said to be in row $i$ and column $j$. Consider a fault-free tiling of the $m\times n$ rectangle for $n>k$. Not all cells in column 1 are covered by horizontal tiles since this would create a fault at $x=k$. It follows that there is precisely one vertical tile in the first column (see Figure \ref{fig:placed}); suppose the topmost cell of this tile is $(1,b)$. 
  \begin{figure}[h]
\centering
  \includegraphics[scale=1]{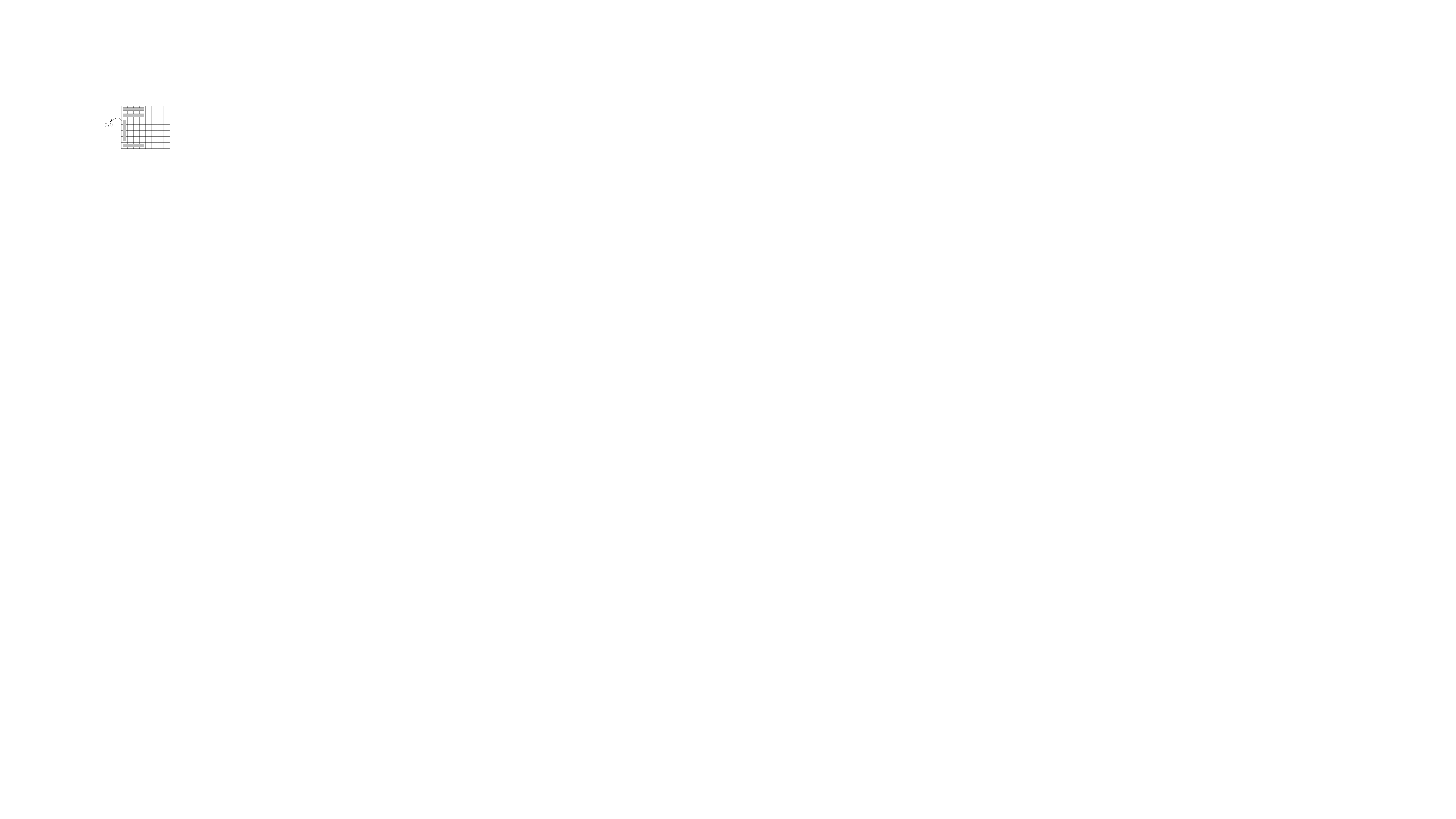} 
  \caption{Vertical tile in the first column.}     
  \label{fig:placed}
\end{figure}

We claim that the $m-k$ rows that do not intersect this vertical tile consist of only horizontal tiles. If not, then consider the leftmost vertical tile that intersects one of the aforementioned rows and suppose its top cell is $(a',b')$ where $b'\neq b$. By reflecting through a horizontal line if necessary, we may assume without loss of generality that $b'>b$ (see Figure \ref{fig:contra}).
\begin{figure}[h]
\centering
\includegraphics{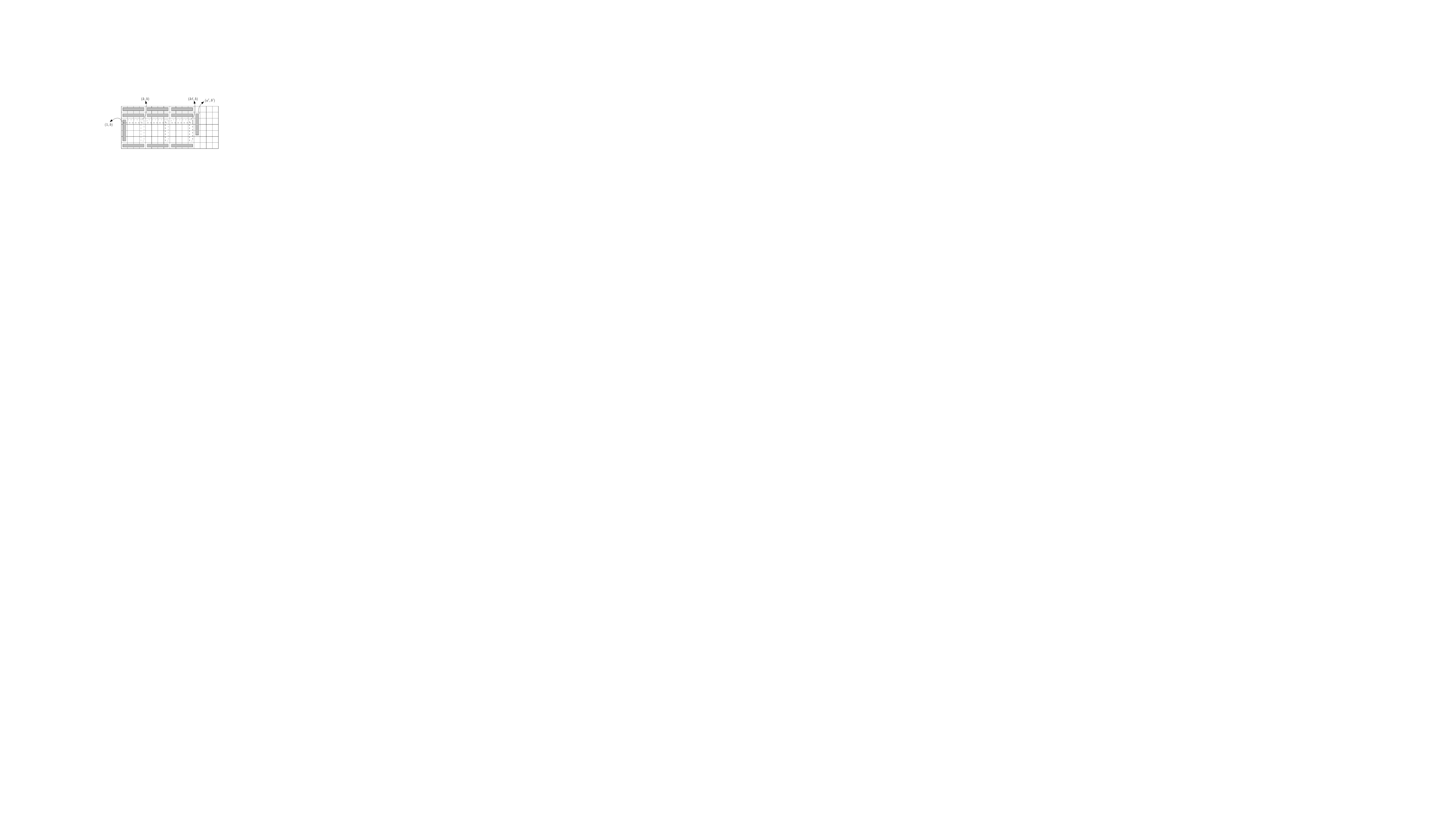} 
\caption{The contradiction obtained.}      
\label{fig:contra}
\end{figure}
By the minimality of $a'$, it follows that $a'=k\ell+1$ for some positive integer $\ell$. Moreover, the first $k\ell$ cells in each row that does not intersect the vertical tile in the first column are covered by $\ell$ horizontal tiles. Now consider the tile covering the cell $(k\ell,b)$. If this tile were vertical, there would be a fault at $x=k\ell$ and, therefore, this tile must be horizontal. By the same reasoning, the tile covering the cell $(k(\ell-1),b)$ must be horizontal. Continuing this line of reasoning, it is clear that the tile covering the cell $(k,b)$ must be horizontal which is impossible since this tile would then cover $(1,b)$ which is already covered by the vertical tile in the first column. This proves the claim and the lemma. 
\end{proof}
The idea used to prove Lemma \ref{lem:main} can also be used to show that for $n>k$, any fault-free tiling of an $m\times n$ rectangle in which each tile has dimensions $k\times j$ for some $1\leq j\leq k$ has $m-k$ rows consisting of only horizontal tiles but we do not require this stronger result here.

\begin{proposition}
  \label{prop:fft}
  Let $k>1$ be a positive integer and suppose $k<m<2k$. The number of fault-free tilings of an $m\times k\ell$ rectangle is given by
  \begin{align*}
    a(m,k\ell)=
    \begin{cases}
      m-k+2 & \ell=1,\\
      (m-k+1){k+\ell-3 \choose k-2} & \ell \geq 2.
    \end{cases}
  \end{align*}
\end{proposition}
\begin{proof}
  For an $m\times k$ rectangle there is one fault-free tiling in which all tiles are horizontal. Any other fault-free tiling of this rectangle contains precisely $m-k$ horizontal tiles and $k$ vertical tiles; there are $m-k+1$ such tilings. Thus $a(m,k)=m-k+2$.  

  Now consider a fault-free tiling of an $m\times k\ell$ rectangle. Such a tiling has $m-k$ horizontal rows by Lemma \ref{lem:main}. If these $m-k$ rows are removed, what remains is a tiling of a $k\times k\ell$ rectangle in which no faults occur at $x=jk$ for $j\geq 1$; denote by $N$ the number of tilings so obtained. By Remark~\ref{rem:compositions}, these tilings are in bijection with compositions $(n_1,\ldots,n_r)$ of $k\ell$ where each $n_i\in \{1,k\}$ and such that $k$ does not divide $\sum_{i=1}^sn_i$ for each $s<r$. This condition on the partial sums implies $n_1=1=n_r$. Further, the number of $n_i$'s equal to 1 must be a multiple of $k$ but the condition on the partial sums ensures that no more than $k$ of them can equal 1. It follows that precisely $k$ of the $n_i$ are 1 and hence $r=k+\ell-1$.  Since precisely $k-2$ of the $n_i(2\leq i\leq r-1)$ are equal to 1 and all these choices are possible, we obtain $N={k+\ell-3 \choose k-2}.$ On the other hand, the number of fault-free tilings of an $m\times k\ell$ rectangle that yield a given $k\times k\ell$ rectangle upon removing the $m-k$ horizontal rows is clearly $m-k+1$. Therefore
  \begin{align*}
  a(m,k\ell)&=(m-k+1){k+\ell-3 \choose \ell-1}.  \qedhere
  \end{align*}
\end{proof}

\begin{corollary}
  \label{cor:fft}
Under the hypotheses of Proposition \ref{prop:fft} every fault-free tiling of an $m\times k\ell$ rectangle ($\ell>1$) by $k\times 1$ tiles contains precisely $k$ vertical tiles.
\end{corollary}
\begin{remark}
  \label{rem:blocks}
It is clear from the proof of Proposition \ref{prop:fft} that each fault-free tiling of an $m\times k\ell$ rectangle ($\ell>1; k<m<2k$) contains precisely $\ell-1$ `blocks', where a block is defined as a collection of $k$ contiguous horizontal $k\times 1$ tiles, one on top of the other.
\end{remark}

\begin{theorem}
  \label{th:main1}
Suppose $k<m<2k$ and $h_k(m,n)$ denotes the number of tilings of an $m\times n$ rectangle with $k\times 1$ tiles. Then 
  \begin{align*}
\sum_{n\geq 0}h_k(m,n)x^n=     \frac{(1-x^k)^{k-1}}{(1-x^k)^k - (m-k+1)x^k}.
  \end{align*}
\end{theorem}
\begin{proof}
  By Proposition~\ref{prop:fft}, it follows that the generating function for fault-free tilings is given by
  \begin{align}
    A(x)&=\sum_{\ell\geq 1}a(m,k\ell)x^{k\ell}\nonumber \\
    &=(m-k+2)x^k+(m-k+1)\sum_{\ell\geq 2}{k+\ell-3 \choose \ell-1}x^{k\ell} \nonumber \\ 
    &=x^k + \frac{(m-k+1)x^k}{(1-x^k)^{k-1}}.\label{eq:a}
  \end{align} 
The theorem now follows from Lemma \ref{lem:gfs}.
\end{proof}

  Several OEIS sequences which correspond to the generating function in Theorem \ref{th:main1} are shown in Table \ref{tab:oeis}.
  \begin{center}
    \begin{table}[h]
    \begin{tabular}{|c|c|c|}
      \hline   
      $k$ & $m$ &  \text{OEIS entry}\\
      \hline
    2 & 3 &\href{http://oeis.org/A001835}{A001835}\\
    3& 4 & \href{http://oeis.org/A049086}{A049086}\\
    3& 5& \href{http://oeis.org/A236576}{A236576} \\
    4& 5 & \href{http://oeis.org/A236579}{A236579}\\
    4& 6 & \href{http://oeis.org/A236580}{A236580}\\
    4& 7 & \href{http://oeis.org/A236581}{A236581}\\
    \hline
    \end{tabular}
    \caption{OEIS entries}
    \label{tab:oeis}
    \end{table} 
  \end{center}
Theorem \ref{th:main1} may be viewed as a special case of the following more general result.    
  \begin{theorem}\label{th:vert}
    Suppose $k<m<2k$ and let $b_k(m,n,r)$ denote the number of $k\times 1$ tilings of an $m\times n$ rectangle which contain precisely $r$ vertical tiles. Then
    \begin{align}
      \sum_{n,r\geq 0}b_k(m,n,r)x^ny^r=  \frac{(1-x^k)^{k-1}}{(1-x^k)^k - (m-k+1)x^ky^k}.\label{eq:bk}
    \end{align}
  \end{theorem}
  \begin{proof}
Consider the number of vertical tiles in fault-free tilings of an $m\times k\ell$ rectangle by $k\times 1$ tiles. For $\ell=1$, there is one such tiling with no vertical tiles and $m-k+1$ tilings with precisely $k$ vertical tiles. By Corollary \ref{cor:fft} each such fault-free tiling for $\ell>1$ contains precisely $k$ vertical tiles. Therefore the generating function for $b_k(m,n,r)$ is
$$\frac{1}{1-A(x,y)},$$
where, $A(x,y)$ can be computed from the expression for $A(x)$ in Equation \eqref{eq:a} as
\begin{align*}
  A(x,y)&=x^k + \frac{(m-k+1)x^ky^k}{(1-x^k)^{k-1}}.\qedhere
\end{align*}
\end{proof}
\section{Brick tilings of a cuboid} 
The results of the previous section can be used to derive the generating function for the number of tilings of an $m\times n\times k$ cuboid with $k\times k\times 1$ bricks. In order to prove the result we will require the following theorem on tilings of an $m\times n$ rectangle with $k\times 1$ and $k\times k$ tiles. Fault-free tilings for $m\leq k$ are easily enumerated, so we consider the case $m>k$. 
\begin{theorem}
  \label{th:main2}
  Suppose $k<m<2k$ and $h'(m,n)$ denotes the number of tilings of an $m\times n$ rectangle with $k\times 1$ tiles and $k\times k$ tiles. Then
  \begin{align*}
\sum_{n\geq 0}h'(m,n)x^n = \frac{(1-2x^k)^{k-1}}{(1-2x^k)^{k-1}[1 - (m-k+2)x^k] - (m-k+1)x^k}.
  \end{align*}
\end{theorem}
\begin{proof}
  Let $a'(m,n)$ denote the number of fault-free tilings of an $m\times n$ rectangle with $k\times 1$ and $k\times k$ tiles. For fault-free tilings of an $m\times k$ rectangle, we have
  \begin{enumerate}
  \item 1 tiling with only horizontal $k\times 1$ tiles;
  \item $m-k+1$ tilings containing a $k\times k$ tile;
  \item $m-k+1$ tilings containing precisely $k$ vertical $k\times 1$ tiles.
  \end{enumerate}
  Thus $a'(m,k)=2m-2k+3$. Now suppose $n=k\ell$ with $\ell>1$.
  Let $T'_k(m,n)$ be the set of all fault-free tilings of an $m\times n$ rectangle by $k\times 1$ and $k\times k$ tiles. Denote by $T_k(m,n)$ the subset of $T'_k(m,n)$ consisting of fault-free tilings of an $m\times n$ rectangle with only $k\times 1 $ tiles. To each tiling $T'\in T'_k(m,n)$ we can associate a new tiling by replacing each $k\times k$ tile in $T'$ by $k$ horizontal $k\times 1$ tiles (see Figure~\ref{fig:replace}); it is easily seen that this new tiling is fault-free, and therefore lies in $T_k(m,n)$.
\begin{figure}[h]
\centering
\includegraphics{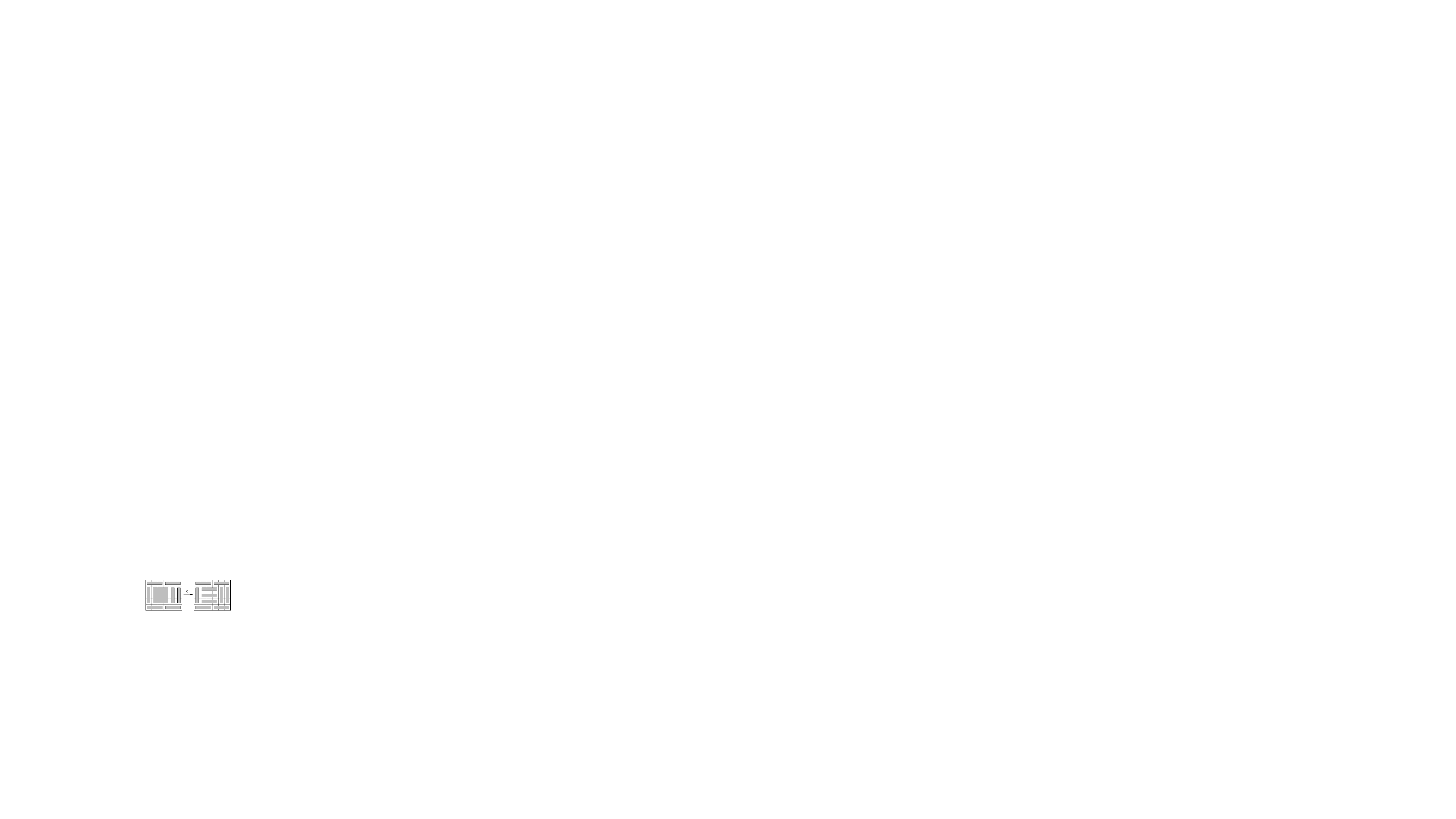} 
\caption{A fault-free tiling and its image under $\Psi$.}   
  \label{fig:replace}
\end{figure}
  This correspondence gives a map $\Psi :T'_k(m,n)\to T_k(m,n)$ which is clearly surjective. In fact the fibers of $\Psi$ are of cardinality $2^{\ell-1}$. To see this note that each $T\in T_k(m,n)$ has precisely $\ell-1$ blocks as defined in Remark \ref{rem:blocks}. The tilings in $\Psi^{-1}(T)$ are precisely those which can be obtained from $T$ by optionally replacing each of the blocks in $T$ by a $k\times k$ tile, for a total of $2^{\ell-1}$ choices. In summary, we have
\begin{align*}
  a'(m,k\ell)=
  \begin{cases}
    2m-2k+3 & \ell=1,\\
    2^{\ell-1}(m-k+1){\ell+k-3 \choose \ell-1} & \ell>1. 
  \end{cases}
\end{align*}
The generating function for fault-free tilings is therefore
\begin{align*}
  A(x) = (m-k+2)x^k + (m-k+1) \frac{x^k}{(1-2x^k)^{k-1}}.
\end{align*}
The theorem now follows from Lemma \ref{lem:gfs}.
\end{proof}
The next result extends the above theorem by accounting for the number of tiles of a given type and orientation.
\begin{theorem} \label{th:vert2}
    Suppose $k<m<2k$ and let $c_k(m,n,r,s)$ denote the number of tilings of an $m\times n$ rectangle with $k\times 1$ and $k\times k$ tiles which contain precisely $r$ vertically placed $k\times 1$ tiles and $s$ square tiles. Then
    \begin{align}
      \sum_{n,r,s\geq 0}c_k(m,n,r,s)x^ny^rz^s=\hspace*{3in} \nonumber\\ \frac{(1-x^k-x^kz)^{k-1}}{\left(1-x^k-(m-k+1)x^kz\right)(1-x^k-x^kz)^{k-1}-(m-k+1)x^ky^k}.\label{eq:bigeq}
    \end{align}  
  \end{theorem}
  \begin{proof}
    We argue as in the proof of Theorem \ref{th:vert}. If $C(x,y,z)$ denotes the generating function for $c_k(m,n,r,s)$, then  
    $$
C(x,y,z)=\frac{1}{1-A(x,y,z)},
$$
where $A(x,y,z)$ is computed from the discussion in the proof of Theorem \ref{th:main2} as
\begin{align*}
  A(x,y,z)&= x^k(1+(m-k+1)z+(m-k+1)y^k)\\
  &\qquad+\sum_{\ell\geq 2}(m-k+1){\ell+k-3 \choose \ell-1}x^{k\ell}y^k(1+z)^{\ell-1}\\
  &=x^k(1+(m-k+1)z)\\
             &\qquad +\sum_{\ell\geq 1}(m-k+1){\ell+k-3 \choose \ell-1}y^kx^{k\ell}(1+z)^{\ell-1}\\
    &=x^k(1+(m-k+1)z)+\frac{(m-k+1)x^ky^k}{(1-x^k-x^kz)^{k-1}}.\qedhere
\end{align*}
\end{proof}
 Note that if we set $z=0$ in Theorem~\ref{th:vert2} then we obtain Theorem~\ref{th:vert} while the substitution $y=z=1$ yields Theorem \ref{th:main2}.
\begin{theorem}
  \label{th:3d}
  Suppose $k<m<2k$ and let $\bar{h}(m,n)$ denote the number of tilings of an $m\times n\times k$ cuboid with $k\times k\times 1$ bricks. Then
  \begin{align*}
\sum_{n\geq 0}\bar{h}(m,n)x^n = \frac{(1-2x^k)^{k-1}}{(1-2x^k)^{k-1}[1 - (m-k+2)x^k] - (m-k+1)x^k}.
  \end{align*}
\end{theorem}
\begin{proof}
  Consider a cuboid in the positive octant with a corner at the origin and with its sides of length $n,m$ and $k$ along the $x,y$ and $z$-axes respectively. Let $\tau$ be a tiling of this cuboid by $k\times k\times 1$ bricks. The bricks of the tiling which touch the $xy$-plane determine a tiling $\tau'$ of an $m\times n$ rectangle with $k\times 1$ and $k\times k$ tiles as shown in Figure \ref{fig:3dproj}.  
  \begin{figure}[!h]
    \centering
    \includegraphics{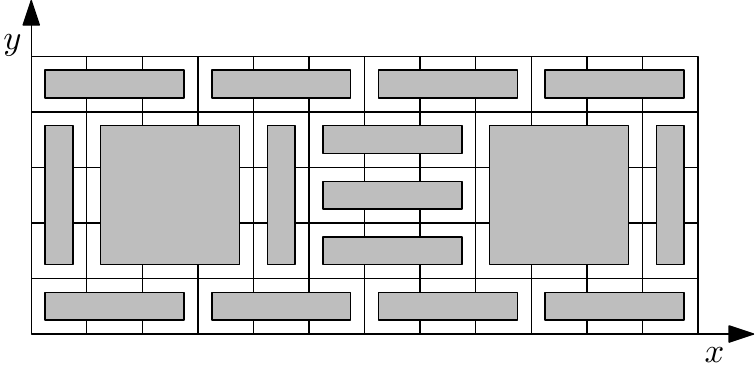}  
    \caption{A tiling $\tau'$ for $m=5,n=12$ and $k=3$.}
    \label{fig:3dproj}
  \end{figure}
 In fact $\tau$ is uniquely determined by $\tau'$ as follows. The tiles in $\tau'$ can be seen to correspond to the projections of the bricks in $\tau$ onto the $xy$-plane: each $k\times 1$ tile of $\tau'$ is the projection of a single brick of $\tau$ while each $k\times k$ tile of $\tau'$ is the projection of precisely $k$ bricks in $\tau$. This gives a one-one correspondence between tilings of the cuboid by $k\times k\times 1$ bricks and tilings of an $m\times n$ rectangle by tiles of size $k\times 1$ and $k\times k$. The result now follows from Theorem~\ref{th:main2}. 
\end{proof}
Replacing $z$ by $z^k$ in the generating function \eqref{eq:bigeq}, the following result is obtained.
\begin{corollary}\label{cor:vert3d}
    Suppose $k<m<2k$ and let $d_k(m,n,r,s)$ denote the number of tilings of an $m\times n\times k$ cuboid with $k\times k\times 1$ bricks which contain precisely $r$ bricks parallel to the $yz$-plane and $s$ bricks parallel to the $xy$-plane. Then
    \begin{align*}
      \sum_{n,r,s\geq 0}d_k(m,n,r,s)x^ny^rz^s=\hspace*{3in} \nonumber\\ \frac{(1-x^k-x^kz^k)^{k-1}}{\left(1-x^k-(m-k+1)x^kz^k\right)(1-x^k-x^kz^k)^{k-1}-(m-k+1)x^ky^k}.\label{eq:bigeq}
    \end{align*}  
\end{corollary}
\section{Acknowledgements}
The second author was partially supported by a MATRICS grant MTR/2017/000794 awarded by the Science and Engineering Research Board.
\bibliographystyle{plain}

\end{document}